\documentclass[11pt,a4paper,leqno]{amsart}

\usepackage[latin1]{inputenc}
\usepackage[T1]{fontenc}
\usepackage{amsfonts}
\usepackage{amsmath}
\usepackage{amssymb}
\usepackage{eurosym}
\usepackage{mathrsfs}
\usepackage{palatino}
\usepackage{color}
\usepackage{esint}
\usepackage{url}
\usepackage{verbatim}
\usepackage{graphicx}

\usepackage{enumerate}

\usepackage[pagebackref,hypertexnames=false, colorlinks, citecolor=blue, linkcolor=blue, urlcolor=red]{hyperref}

\newcommand{\boxx}{\scalebox{0.5}{$\square$}}

\newcommand{\R}{\mathbb{R}}
\newcommand{\C}{\mathbb{C}}

\newcommand{\Z}{\mathbb{Z}}
\newcommand{\E}{\mathbb{E}}

\newcommand{\bbP}{\mathbb{P}}

\numberwithin{equation}{section}

\newcommand{\ud}[0]{\,\mathrm{d}}

\newcommand{\esssup}[0]{\operatornamewithlimits{ess\,sup}}

\newcommand{\dist}[0]{\operatorname{dist}}

\newcommand{\abs}[1]{|#1|}

\newcommand{\Norm}[2]{\|#1\|_{#2}}

\newcommand{\BNorm}[2]{\Big\|#1\Big\|_{#2}}
\newcommand{\pair}[2]{\langle #1,#2 \rangle}

\newcommand{\ave}[1]{\langle #1\rangle}

\newcommand{\bddlin}[0]{\mathcal{L}}

\newcommand{\BMO}[0]{\operatorname{BMO}}

\newcommand{\calD}[0]{\mathcal{D}}

\swapnumbers
\theoremstyle{plain}
\newtheorem{thm}[equation]{Theorem}
\newtheorem{lem}[equation]{Lemma}
\newtheorem{prop}[equation]{Proposition}
\newtheorem{cor}[equation]{Corollary}

\theoremstyle{definition}
\newtheorem{defn}[equation]{Definition}

\theoremstyle{remark}
\newtheorem{rem}[equation]{Remark}

\title{Exotic Calder\'on--Zygmund operators}

\author{Tuomas Hyt\"onen}
\author{Kangwei Li}
\author{Henri Martikainen}
\author{Emil Vuorinen}

\address[T.H. \& E.V.]{Department of Mathematics and Statistics, University of Helsinki, P.O.B. 68, FI-00014 University of Helsinki, Finland}
\email{tuomas.hytonen@helsinki.fi}
\email{emil.vuorinen@helsinki.fi}

\address[K.L.]{Center for Applied Mathematics, Tianjin University, Weijin Road 92, 300072 Tianjin, China}
\email{kli@tju.edu.cn}

\address[H.M.]{Department of Mathematics and Statistics, Washington University in St. Louis, 1 Brookings Drive, St. Louis, MO 63130, USA}
\email{henri@wustl.edu}

\makeatletter
\@namedef{subjclassname@2020}{%
  \textup{2020} Mathematics Subject Classification}
\makeatother

\subjclass[2020]{42B20}
\keywords{singular integrals, multi-parameter analysis, Zygmund dilations, multiresolution analysis}

\begin{document}

\allowdisplaybreaks

\begin{abstract}
  We study singular integral operators with kernels that are
  more singular than standard Calder\'on--Zygmund kernels, but less singular than bi-parameter
  product Calder\'on--Zygmund kernels. These kernels arise as restrictions to two dimensions of certain three-dimensional kernels adapted to so-called Zygmund dilations, which is part of our motivation for studying these objects.
  We make the case that such kernels can, in many ways, be seen as part of the extended realm of standard kernels by 
  proving that they satisfy both a $T1$ theorem and commutator estimates in a form reminiscent of the corresponding results for standard Calder\'on--Zygmund kernels.
  However, we show that one-parameter weighted estimates, in general, fail. 
\end{abstract}

\maketitle

\section{Introduction}
Working on the Euclidean product space $\R^2 = \R \times \R$, we define for $x = (x^1, x^2)$ and $y = (y^1, y^2)$
the decay factor
\begin{equation}\label{eq:decfac}
  D_{\theta}(x, y) := \Bigg(\frac{|x^1-y^1|}{|x^2-y^2|} + \frac{|x^2-y^2|}{|x^1-y^1|}\Bigg)^{-\theta} < 1, \qquad 
  \theta \in (0, 1],
\end{equation}
whenever $x^1 \ne y^1$ and $x^2 \ne y^2$.
Notice that this decay factor becomes larger and larger as $\theta$ shrinks. The point is that when $\theta = 1$ it is at its
smallest, and then
$$
\frac{1}{|x^1-y^1|}\frac{1}{|x^2-y^2|} D_{1}(x,y) = \frac{1}{|x^1-y^1|^{2}+|x^2-y^2|^{2}} =
\frac{1}{|x-y|^2}.
$$
That is, in this case the bi-parameter size estimate multiplied with this decay factor yields the usual
one-parameter size estimate. When $\theta < 1$ the decay factor is larger and the corresponding product
is something between the bi-parameter and one-parameter size estimate.

We say that kernels that decay like $$\frac{1}{|x^1-y^1|}\frac{1}{|x^2-y^2|} D_{\theta}(x,y)$$
for some $\theta$ and satisfy some similar continuity estimates are $CZX$ kernels -- one can pronounce the ``$X$'' in ``$CZX$'' as ``exotic''.
Such kernels are more singular than the standard Calder\'on--Zygmund kernels, but less singular than the product Calder\'on--Zygmund(--Journ\'e) kernels \cite{FS,Jo2, Ma1}.
Even with $\theta = 1$ they are different from the standard Calder\'on--Zygmund kernels -- in this case the difference is only in the
H\"older estimates (see Section \ref{sec:2}).
The $CZX$ kernels can, for example, be motivated by looking at Zygmund dilations \cite{FP, MRS, NW, RS}. Zygmund dilations are a group of dilations lying in between the standard product theory
and the one-parameter setting -- in $\R^3 = \R \times \R^2$
they are the dilations $(x_1, x_2, x_3) \mapsto (\delta_1 x_1, \delta_2 x_2, \delta_1 \delta_2 x_3)$. 
Recently, in \cite{HLLT} and subsequently in \cite{DLOPW2,HLLTW} general convolution form singular integrals invariant under Zygmund dilations were
studied. In these papers the decay factor
$$
t \mapsto \Big(t+\frac1t\Big)^{-\theta}
$$
controls the additional, compared to the product setting, decay with respect to the Zygmund ratio
$$
\frac{\abs{x_1x_2}}{\abs{x_3}}.
$$
See also our recent paper \cite{HLMV:ZYG} which attacks the Zygmund setting from the point of view of new 
multiresolution methods.
Essentially, in the current paper we isolate the conditions on the lower-dimensional kernels obtained by fixing the variables $x^1,y^1$ in the Zygmund setting \cite{HLLT, HLMV:ZYG}
and ignoring the dependence on these variables. A class of $CZX$ operators is also induced by the Fefferman-Pipher multipliers
\cite{FP} -- importantly, they satisfy $\theta = 1$ but with an additional logarithmic growth factor.
This subtle detail has a key relevance for the weighted estimates as we next explain. 

It has been known for a long time that Calder\'on--Zygmund operators act boundedly in the weighted spaces $L^p(w)$ whenever $w$ belongs to the Muckenhoupt class $A_p$, defined by the finiteness of the weight constant
\begin{equation*}
  [w]_{A_p}:=\sup_J \ave{w}_J\ave{w^{-1/(p-1)}}_J^{p-1},
\end{equation*}
where the supremum is over all cubes $J$. On the other hand, the more singular multi-parameter Calder\'on--Zygmund(--Journ\'e) operators in general satisfy such bounds only for the smaller class of strong $A_p$ weights, defined by $[w]_{A_p^*}$, where the supremum is over all axes-parallel rectangles.
While on a general level the $CZX$ operators
behave quite well with any $\theta$, even with $\theta < 1$, for one-parameter weighted estimates it is critical that $\theta = 1$,
the aforementioned logarithmic extra growth being allowed. 
\begin{thm}
  Let $T\in\bddlin(L^2(\R^2))$ be an operator with a CZX kernel.
  \begin{enumerate}
    \item If $\theta< 1$ in \eqref{eq:decfac}, one-parameter weighted estimates may fail.
    \item  If $\theta = 1$ in \eqref{eq:decfac}, possibly with a logarithmic growth factor,
    then for every $p\in(1,\infty)$ and every $w\in A_p(\R^2)$ the operator $T$ extends boundedly to $L^p(w)$.
  \end{enumerate}
\end{thm}
In the paper \cite{HLMV:ZYG} we also develop the corresponding counterexamples in the full Zygmund case.
There the interest is whether Zygmund singular integrals are weighted bounded with respect to the
Zygmund weights -- a larger class than the strong $A_p$ with the supremum
running only over the so-called Zygmund rectangles satisfying the natural scaling. For $\theta < 1$ the situation
parallels the one from the $CZX$ world -- they need not be weighted bounded with respect to the Zygmund weights.

Apart from the weighted estimates, we want to make the case that, in many ways, the $CZX$ kernels with an arbitrary $\theta$ can be seen as part of the extended realm of standard kernels, rather than the more complicated product theory. In particular, the $T1$ theorem for $CZX$ kernels takes the following form
reminiscent of the standard $T1$ theorem \cite{DJ}.
\begin{thm}\label{thm:T1CZX}
  Let $B(f,g)$ be a bilinear form defined on finite linear combinations of indicators of cubes of $\R^2$, and such that
  \begin{equation*}
    B(f,g)=\iint K(x,y)f(y)g(x)\ud x\ud y
  \end{equation*}
  when $\{f\neq 0\}\cap\{g\neq 0\}=\varnothing$, where $K\in CZX(\R^2)$. Then the following are equivalent:
  \begin{enumerate}
    \item There is a bounded linear $T\in\bddlin(L^2(\R^2))$ such that $\pair{Tf}{g}=B(f,g)$.
    \item $B$ satisfies
    \begin{itemize}
      \item the weak boundedness property $\abs{B(1_I,1_I)}\lesssim\abs{I}$ for all cubes $I\subset\R^2$, and
      \item the $T(1)$ conditions
  \begin{equation*}
     B(1,g)=\int b_1 g,\qquad B(f,1)=\int b_2 f
  \end{equation*}
  for some $b_1,b_2\in\BMO(\R^2)$ and all $f,g$ with $\int f=0=\int g$.
    \end{itemize}
  \end{enumerate}
  Moreover, under these conditions,
  \begin{itemize}
    \item[\it (3)]\label{it:CZX-BMO} $T$ defines a bounded operator from $L^\infty(\R^2)$ to $\BMO(\R^2)$, from $L^1(\R^2)$ to $L^{1,\infty}(\R^2)$, and on $L^p(\R^2)$ for every $p\in(1,\infty)$.
  \end{itemize}
  \end{thm}
  In fact, our proof also gives a representation of $B(f, g)$, Theorem \ref{thm:rep}, which includes
  both one-parameter \cite{Hy} and bi-parameter \cite{Ma1} elements.
  The following commutator bounds follow from the representation, however, the argument
  is not entirely standard due to the hybrid nature of the model operators.
  \begin{thm}
    Let $T\in\bddlin(L^2(\R^2))$ be an operator associated with a CZX kernel $K$. Then we have
    $$
    \| [b,T]f \|_{L^p} \lesssim \|b\|_{\BMO} \| f \|_{L^p}
    $$
    whenever $p \in (1, \infty)$. Here $[b,T]f:=bTf-T(bf)$.
   \end{thm}
   Thus, the commutator estimate holds with the one-parameter $\BMO$ space. This is another
   purely one-parameter feature of these exotic operators. As the weighted estimates
   do not, in general, hold, the commutator estimate cannot be derived from the well-known Cauchy
   integral trick.
   
   Over the past several years, a standard approach to weighted norm inequalities has been via the methods of sparse domination pioneered by Lerner.
   For $\theta = 1$ we can derive our weighted estimates directly from our representation theorem. However, we also provide
   some additional sparse estimates that give a solid quantitative dependence on the $A_p$ constant
   and yield two-weight commutator estimates for free.
   \begin{thm}\label{thm:quantitative}
    Let $T\in\bddlin(L^2(\R^2))$ be an operator with a CZX kernel with $\theta = 1$.
    Then for every $p\in(1,\infty)$ and every $w\in A_p(\R^2)$ the operator $T$ extends boundedly to $L^p(w)$ with norm
    \begin{equation*}
      \Norm{T}{\bddlin(L^p(w))}\lesssim_p [w]_{A_p}^{p'}.
    \end{equation*}
    Moreover, if
    $\nu=w^{\frac 1p}\lambda^{-\frac 1p}$ with $w,\lambda\in A_p$ and
    \[
      \|b\|_{\BMO_\nu}:= \sup_I \frac 1{\nu(I)}\int_I |b-\ave{b}_I| < \infty,
    \]
    where the supremum is over cubes $I \subset \R^2$, then
    \[
    \|[b,T]\|_{L^p(w)\to L^p(\lambda)}\lesssim \|b\|_{\BMO_\nu}.
    \]
   \end{thm}
  The quantitative bound (in particular quadratic in $[w]_{A_2}$ when $p=2$) is worse than the linear
  $A_2$ theorem valid for classical Calder\'on--Zygmund operators \cite{Hy}.

\subsection*{Acknowledgements}
T. Hyt\"onen and E. Vuorinen were supported by the Academy of Finland through project numbers 314829 (both) and 346314 (T.H.),
and by the University of Helsinki internal grants for the Centre of Excellence in Analysis and Dynamics Research (E.V.) and the Finnish Centre of Excellence in Randomness and Structures ``FiRST'' (T.H.).
K. Li was supported by the National Natural Science Foundation of China through project number 12001400.

\section{CZX kernels}\label{sec:2}
We work in $\R^2 = \R \times \R$.
Let $\theta_1, \theta_2 \in (0, 1]$.
For $x^1 \ne y^1$ and $x^2 \ne y^2$ define
$$
D_{\theta_2}(x, y) := \Bigg(\frac{|x^1-y^1|}{|x^2-y^2|} + \frac{|x^2-y^2|}{|x^1-y^1|}\Bigg)^{-\theta_2} < 1.
$$
We assume that the kernel $K \colon \R^2 \setminus \{x^1 = y^1 \textup{ or } x^2 = y^2\} \to \C$ satisfies
the size estimate
$$
|K(x,y)| \lesssim \frac{1}{|x^1-y^1|}\frac{1}{|x^2-y^2|}  D_{\theta_2}(x,y)
$$
and the mixed H\"older and size estimate
$$
|K(x,y) - K((w^1, x^2), y)| \lesssim \frac{|x^1-w^1|^{\theta_1}}{|x^1-y^1|^{1+\theta_1}} \frac{1}{|x^2-y^2|}  D_{\theta_2}(x,y)
$$
whenever $|x^1-w^1| \le |x^1-y^1|/2$, together with the other three symmetric mixed H\"older and size estimates.
If this is the case, we say that $K\in CZX(\R^2)$. Again,
such kernels are more singular than standard Calder\'on--Zygmund kernels, but less singular than the product Calder\'on--Zygmund(--Journ\'e) kernels.
See Remark \ref{rem:log} for some additional logarithmic factors when $\theta = 1$ and why they are
relevant from the point of view of Fefferman-Pipher multipliers \cite{FP}.

\begin{lem}\label{lem:CZXinterval}
  Let $K\in CZX(\R^2)$ and $x^1, x^2, y^2 \in \R$. Then we have
  $$
  \int_{\R} |K(x,y)|\ud y^1 \lesssim \frac{1}{|x^2-y^2|}.
  $$
  We also have for $L > 0$ that
  $$
  \int_{\{y^1\colon |x^1-y^1| \lesssim L\}} |K(x,y)|\ud y^1 \lesssim \frac{L^{\theta_2}}{|x^2-y^2|^{1+\theta_2}},
  $$
  which is a useful estimate if $L \lesssim |x^2-y^2|$.
\end{lem}
\begin{proof}
  By elementary calculus
  \begin{equation*}
    \begin{split}
      \int_{\R} &|K(x,y)|\ud y^1
      \lesssim \frac{1}{\abs{x^2-y^2}} \int_0^\infty\frac{1}{u}\Big(\frac{u}{\abs{x^2-y^2}}+\frac{\abs{x^2-y^2}}{u}\Big)^{-\theta_2}\ud u \\
      &\lesssim \frac{1}{\abs{x^2-y^2}} \Big( \int_0^{\abs{x^2-y^2}}\frac{\ud u}{u^{1-\theta_2}\abs{x^2-y^2}^{\theta_2}}
        +\int_{\abs{x^2-y^2}}^\infty \frac{\ud u}{u^{1+\theta_2}\abs{x^2-y^2}^{-\theta_2}}\Big)
        \lesssim \frac{1}{\abs{x^2-y^2}},
    \end{split}
    \end{equation*}
    and the logic for the second estimate is also clear from this.
\end{proof}
The first sharper estimate in the next lemma is only needed to derive the weighted estimates
in the case $\theta = 1$.
\begin{lem}\label{lem:CZXaway}
  Let $K\in CZX(\R^2)$ and $J=J^1\times J^2\subset\R^2$ be a square with centre $c_J=(c_{J^1},c_{J^2})$. If $x\in J$ and $y\in(3J^1)^c\times(3J^2)^c$, then
  \begin{equation*}
  \begin{split}
    \abs{K(x,y)-K(c_J,y)} &\lesssim
    \prod_{i=1}^2\frac{1}{\dist(y^i,J^i)}\times \frac{\ell(J)^{\theta_1}(\min_{i=1,2}\dist(y^i,J^i))^{\theta_2-\theta_1}}{(\max_{i=1,2}\dist(y^i,J^i))^{\theta_2}}  \\
    &\lesssim\prod_{i=1}^2\frac{\ell(J)^\theta}{\dist(y^i,J^i)^{1+\theta}},\qquad\theta:=\frac12\min(\theta_1,\theta_2).
\end{split}
  \end{equation*}
  \end{lem}
  
  \begin{proof}
  We have
  \begin{equation*}
    \abs{K(x,y)-K(c_J,y)}
    \leq\abs{K(x^1,x^2,y)-K(c_{J^1},x^2,y)}+\abs{K(c_{J^1},x^2,y)-K(c_{J^1},c_{J^2},y)}.
  \end{equation*}
  Since $2\abs{x^i-c_{J^i}}\leq\ell(J)\leq\dist(y^i,J^i)\leq\min(\abs{y^i-x^i},\abs{y^i-c_{J^i}})$, we have
  \begin{equation*}
  \begin{split}
    \abs{K(x^1,x^2,y)-K(c_{J^1},x^2,y)}
    &\lesssim\frac{\abs{x^1-c_{J^1}}^{\theta_1}}{\abs{x^1-y^1}^{1+\theta_1}}\frac{1}{\abs{x^2-y^2}}D_{\theta_2}(x, y), \\
    \abs{K(c_{J^1},x^2,y)-K(c_{J^1},c_{J^2},y)}
    &\lesssim\frac{1}{\abs{c_{J^1}-y^1}}\frac{\abs{x^2-c_{J^2}}^{\theta_1}}{\abs{c_{J^2}-y^2}^{1+\theta_1}}D_{\theta_2}(c_J, y).
  \end{split}
  \end{equation*}
  Suppose for instance that $\dist(y^1,J^1)\geq \dist(y^2,J^2)$. Then the sum simplifies to
  \begin{equation*}
    \abs{K(x,y)-K(c_J,y)}
    \lesssim\frac{1}{\dist(y^1,J^1)}\frac{\ell(J)^{\theta_1}}{\dist(y^2,J^2)^{1+\theta_1}}
      \Big(\frac{\dist(y^1,J^1)}{\dist(y^2,J^2)}\Big)^{-\theta_2},
  \end{equation*}
  where further
  \begin{equation*}
  \begin{split}
     \frac{\ell(J)^{\theta_1}}{\dist(y^2,J^2)^{\theta_1}}
      \Big(\frac{\dist(y^1,J^1)}{\dist(y^2,J^2)}\Big)^{-\theta_2}
    &=\frac{\ell(J)^{\theta_1}\dist(y^2,J^2)^{\theta_2-\theta_1}}{\dist(y^1,J^1)^{\theta_2}} \\
    &\leq  \Big(\frac{\ell(J)}{\dist(y^1,J^1)}\Big)^{\min(\theta_1,\theta_2)}
    \leq  \prod_{i=1}^2\Big(\frac{\ell(J)}{\dist(y^i,J^i)}\Big)^{\theta}
\end{split}
  \end{equation*}
  with $\theta:=\frac12\min(\theta_1,\theta_2)$.
  \end{proof}
 
A  combination of the previous two lemmas shows that $CZX$-kernels satisfy the H\"ormander integral condition:
  
\begin{lem}\label{lem:CZXHorm}
Let $K\in CZX(\R^2)$, and $x\in J$ for some cube $J=J^1\times J^2\subset\R^2$ with centre $c_J$. Then
\begin{equation*}
   \int_{(3J)^c}\abs{K(x,y)-K(c_J,y)}\ud y\lesssim 1.
\end{equation*}
\end{lem}
  
\begin{proof}
We have
\begin{equation*}
  (3J)^c=((3J^1)^c\times 3J^2)\cup(3J^1\times(3J^2)^c)\cup((3J^1)^c\times(3J^2)^c),
\end{equation*}
where the first two components on the right hand side are symmetric. For these, we simply estimate
\begin{equation*}
\begin{split}
  \int_{(3J^1)^c \times 3J^2}\abs{K(x,y)}\ud y
  &=\int_{(3J^1)^c}\Big(\int_{3J^2}\abs{K(x,y)}\ud y^2\Big)\ud y^1 \\
  &\lesssim\int_{(3J^1)^c}\frac{\ell(J)^{\theta_2}}{\abs{x^1-y^1}^{1+\theta_2}}\ud y^1\lesssim 1,
\end{split}
\end{equation*}
where the first $\lesssim$ was an application of by Lemma \ref{lem:CZXinterval}. The estimate for $K(c_J,y)$ is of course a special case of this with $x=c_J$.

For the remaining component of the integration domain, we have
\begin{equation*}
\begin{split}
    \int_{(3J^1)^c \times (3J^2)^c}\abs{K(x,y)-K(c_J,y)}\ud y
  &\lesssim\int_{(3J^1)^c \times (3J^2)^c}\prod_{i=1}^2\frac{\ell(J)^\theta}{\dist(y^i,J^i)^{1+\theta}}\ud y \\
  &=\prod_{i=1}^2\int_{(3J^i)^c}\frac{\ell(J)^\theta}{\dist(y^i,J^i)^{1+\theta}}\ud y^i \lesssim 1,
\end{split}
\end{equation*}
where the first $\lesssim$ was an application of by Lemma \ref{lem:CZXaway}.
\end{proof}

At this point, we can already provide a proof of part \eqref{it:CZX-BMO} of Theorem \ref{thm:T1CZX}, which we restate as

\begin{prop}
Let $T\in\bddlin(L^2(\R^2))$ be an operator associated with a CZX kernel $K$. Then $T$ extends boundedly from $L^\infty(\R^2)$ into $\BMO(\R^2)$, from $L^1(\R^2)$ into $L^{1,\infty}(\R^2)$, and from $L^p(\R^2)$ into itself for all $p\in(1,\infty)$.
\end{prop}

\begin{proof}
By Lemma \ref{lem:CZXHorm}, the kernel $K$ satisfies the H\"ormander integral condition; the symmetry of the assumption on $K$ ensures that it also satisfies the version with the roles of the first and second variable interchanged. It is well known that any $L^2(\R^2)$-bounded operator with a H\"ormander kernel satisfies the mapping properties stated in the proposition. (See e.g. \cite[\S I.5]{Stein:book} for the boundedness from $L^1(\R^2)$ into $L^{1,\infty}(\R^2)$, and from $L^p(\R^2)$ into itself for $p\in(1,2)$, and \cite[\S IV.4.1]{Stein:book} for the boundedness from $L^\infty(\R^2)$ into $\BMO(\R^2)$. The latter is formulated for convolution kernels $K(x,y)=K(x-y)$, but an inspection of the proof shows that it extends to the general case with trivial modifications. The case of $p\in(2,\infty)$ can be inferred either by duality (observing that the adjoint $T^*$ satisfies the same assumption) or by interpolation between the $L^2(\R^2)$ and the $L^\infty(\R^2)$-to-$\BMO(\R^2)$ estimates.)
\end{proof}

\section{Haar coefficients of CZX forms}
  
 We recall the weak boundedness property and the $T1$ assumptions, which are just the same as in the classical theory for usual Calder\'on--Zygmund forms.
  
  \begin{defn}
    Let $B(f,g)$ be a bilinear form defined on finite linear combinations of indicators of cubes of $\R^2$, and such that
  \begin{equation*}
    B(f,g)=\iint K(x,y)f(y)g(x)\ud x\ud y
  \end{equation*}
  when $\{f\neq 0\}\cap\{g\neq 0\}=\varnothing$, where $K\in CZX(\R^2)$. We say that $B$ is a $CZX(\R^2)$-form.
  \end{defn}
  \begin{defn}
    A $CZX(\R^2)$-form satisfies the weak boundedness property if $\abs{B(1_I,1_I)}\lesssim\abs{I}$ for all cubes $I\subset\R^2$.
    It satisfies the $T1$ conditions if
  \begin{equation*}
     B(1,g)=\int b_1 g,\qquad B(f,1)=\int b_2 f
  \end{equation*}
  for some $b_1,b_2\in\BMO(\R^2)$ and all $f,g$ with $\int f=0=\int g$. Here
  $$
  \Norm{b}{\BMO} =  \Norm{b}{\BMO(\R^2)} :=\sup_I \frac{1}{|I|} \int_I  \abs{b-\ave{b}_I},
  $$
  where the supremum is over all cubes $I \subset \R^2$ and $\ave{b}_I = \frac{1}{|I|} \int_I b$.  
  \end{defn}
  For an interval $I \subset \R$, we denote by $I_{l}$ and $I_{r}$ the left and right
  halves of the interval $I$, respectively. We define $h_{I}^0 = |I|^{-1/2}1_{I}$ and $h_{I}^1 = |I|^{-1/2}(1_{I_{l}} - 1_{I_{r}})$.
  Let now $I = I^1 \times I^2$ be a cube, and define the Haar function $h_I^{\eta}$, $\eta = (\eta^1, \eta^2) \in \{0,1\}^2$, via
  \begin{displaymath}
    h_I^{\eta} = h_{I^1}^{\eta^1} \otimes h_{I^2}^{\eta^2}.
  \end{displaymath}
  \begin{lem}\label{lem:kernelest}
    Let $B$ be a $CZX(\R^2)$-form satisfying the weak boundedness property.
    Then we have that
    \begin{equation*}
    \begin{split}
      \abs{B(h_I^\beta,h_{J}^\gamma)}
      &\lesssim  \prod_{i=1}^2\Big(\frac{\ell(I)}{\ell(I)+\dist(I^i,J^i)}\Big)\times
      \frac{\ell(I)^{\theta_1}(\ell(I)+\min_{i=1,2}\dist(I^i,J^i))^{\theta_2-\theta_1}}{(\ell(I)+\max_{i=1,2}\dist(I^i,J^i))^{\theta_2}}  \\
      &\lesssim\prod_{i=1}^2\Big(\frac{\ell(I)}{\ell(I)+\dist(I^i,J^i)}\Big)^{1+\theta},\qquad
      \theta:=\frac12\min(\theta_1,\theta_2),
\end{split}
    \end{equation*}
    whenever $I,J$ are dyadic cubes with equal side lengths $\ell(I)=\ell(J)$ and at least $\beta \ne 0$ or $\gamma \ne 0$.
    \end{lem}
    \begin{proof}
      We consider several cases.

\subsubsection*{Adjacent cubes:}
By this we mean that $\operatorname{dist}(I,J)=0$, but $I\neq J$. Here we simply put absolute values inside. We are thus led to estimate
\begin{equation}\label{eq:CZXadj}
  \int_I\int_J\abs{K(x,y)h_I^\beta(x)h_J^\gamma(y)}\ud y \ud x
  \leq\frac{1}{\abs{I}}\int_I\int_J\abs{K(x,y)}\ud y \ud x.
\end{equation}
By symmetry, we may assume for instance that $I^2\neq J^2$. By Lemma \ref{lem:CZXinterval}, we have
\begin{equation*}
  \int_{J^1}\abs{K(x,y)}\ud y^1\lesssim \frac{1}{\abs{x^2-y^2}}.
\end{equation*}
By the assumption that $I^2\neq J^2$ we have 
\begin{equation*}
  \int_{I^2}\int_{J^2}\frac{\ud x^2\ud y^2}{\abs{x^2-y^2}} \le \int_{3J^2 \setminus J^2}\int_{J^2} \frac{\ud x^2\ud y^2}{\abs{x^2-y^2}}
  \lesssim \ell(I).
\end{equation*}
The dependence on $x^1$ has already disappeared, and integration with respect to $x^1\in I^1$ results in another $\ell(I)$. Then we are only left with observing that $\ell(I)^2/\abs{I}=1$.

\subsubsection*{Equal cubes:} Now
\begin{equation*}
  B(h_I^\beta,h_I^\gamma)
   =\sum_{I',J'\in\operatorname{ch}(I) } \ave{h_I^\beta}_{I'}\ave{h_I^\gamma}_{J'} B(1_{I'}, 1_{J'}),
\end{equation*}
where $\abs{\ave{h_I^\beta}_{I'}\ave{h_I^\gamma}_{J'}}=\abs{I}^{-1}$. 
For $J'=I'$, we have $\abs{B(1_{I'},1_{I'})}\lesssim\abs{I'}\leq\abs{I}$ from the WBP.
For $J'\neq I'$, we can estimate the term as in the case of adjacent $I\neq J$, recalling that only the size and no cancellation of the Haar functions was used there.

\subsubsection*{Cubes separated in one direction:} 
By this we mean that, say, $\operatorname{dist}(I^1,J^1)=0<\operatorname{dist}(I^2,J^2)$, or the same with $1$ and $2$ interchanged. We still apply only the non-cancellative estimate \eqref{eq:CZXadj} (in contrast to what one would do with standard Calder\'on--Zygmund operators). By Lemma \ref{lem:CZXinterval}, we have
\begin{equation*}
  \int_{J^1}\abs{K(x,y)}\ud y^1\lesssim\frac{\ell(I)^{\theta_2}}{\abs{x^2-y^2}^{1+\theta_2}}\lesssim \frac{\ell(I)^{\theta_2}}{(\ell(I)+\dist(I^2,J^2))^{1+\theta_2}}.
\end{equation*}
There is no more dependence on the remaining variables $x^1,x^2,y^2$, so integrating over these gives the factor $\ell(I)^3$. After dividing by $\abs{I}=\ell(I)^2$ in \eqref{eq:CZXadj}, we arrive at the bound
\begin{equation*}
  \Big(\frac{\ell(I)}{\ell(I)+\dist(I^2,J^2)}\Big)^{1+\theta_2}.
\end{equation*}

\subsubsection*{Cubes separated in both directions:} 
By this we mean that $\operatorname{dist}(I^i,J^i)>0$ for both $i=1,2$. It is only here that we make use of the assumed cancellation of at least one of the Haar functions, say $h_I^\beta$. Thus
\begin{equation*}
  B(h_I^\beta,h_J^\gamma)
  =\int_I\int_J [K(x,y)-K(c_I,y)]h_I^\beta(x) h_J^\gamma(y)\ud y\ud x,
\end{equation*}
where $c_I=(c_{I^1},c_{I^2})$ is the centre of $I$. Now $x\in I$ and $y^i\in J^i\subset(3I^i)^c$ for $i=1,2$, so Lemma \ref{lem:CZXaway} applies to give
\begin{equation*}
\begin{split}
  &\abs{B(h_I^\beta,h_J^\gamma)} \\
  &\lesssim\int_I\int_J\prod_{i=1}^2\frac{1}{(\ell(I)+\dist(I^i,J^i))}\times\frac{\ell(I)^{\theta_1}(\ell(I)+\min_{i=1,2}\dist(I_i,J_i))^{\theta_2-\theta_1}}{
      (\ell(I)+\max_{i=1,2}\dist(I^i,J^i))^{\theta_2}}\times\frac{1}{\abs{I}}\ud y\ud x,
\end{split}
\end{equation*}
which readily simplifies to the claimed bound after $\abs{I}^2/\abs{I}=\ell(I)^2$.
\end{proof}

\section{Dyadic representation and $T1$ theorem}

Let $\calD_0$ be the standard dyadic grid in $\R$. For $\omega \in \{0,1\}^{\Z}$, $\omega = (\omega_i)_{i \in \Z}$, 
we define the shifted lattice
$$
\calD(\omega) := \Big\{L + \omega := L + \sum_{i\colon 2^{-i} < \ell(L)} 2^{-i}\omega_i \colon L \in \calD_0\Big\}.
$$
Let $\bbP_{\omega}$ be the product probability measure on $\{0,1\}^{\Z}$. 
We recall the notion of $k$-good cubes from \cite{GH}.
We say that $G \in \calD(\omega, k)$, $k \ge 2$,
if $G \in \calD(\omega)$ and
\begin{equation}\label{eq:DefkGood}
d(G, \partial G^{(k)}) \ge \frac{\ell(G^{(k)})}{4} = 2^{k-2} \ell(G).
\end{equation}
Notice that for all $L \in \calD_0$ and $k \ge 2$ we have
\begin{equation}\label{eq1}
  \bbP_{\omega}( \{ \omega\colon L + \omega \in \calD(\omega, k) \})  = \frac{1}{2}.
\end{equation}

For $\sigma = (\sigma^1, \sigma^2) \in \{0,1\}^{\Z} \times \{0,1\}^{\Z}$ and dyadic $\lambda > 0$ define
\begin{align*}
  \calD(\sigma) &:= \calD(\sigma^1) \times \calD(\sigma^2), \\
  \calD_{\lambda}(\sigma) &:= \{I = I^1 \times I^2 \in \calD(\sigma)\colon \ell(I^1) = \lambda \ell(I^2) \}, \\
  \calD_{\boxx}(\sigma) &:= \calD_{1}(\sigma).
\end{align*}
Let $\bbP_{\sigma} := \bbP_{\sigma^1} \times \bbP_{\sigma^2}$. For $k = (k^1, k^2)$, $k^1, k^2 \ge 2$, we define
$\calD(\sigma, k) = \calD(\sigma^1, k^1) \times \calD(\sigma^2, k^2)$.

We will need an estimate for the maximal operator
\begin{equation*}
  M_{\calD_\lambda(\sigma)}f(x):=\sup_{I\in\calD_\lambda(\sigma)}1_I(x)\ave{\abs{f}}_I.
\end{equation*}
Before bounding it, we recall the following interpolation result
due to Stein and Weiss, see \cite[Theorem 2.11]{SW}.
\begin{prop}\label{prop:SW}
Suppose that $1 \le p_0,p_1 \le \infty$ and
let $w_0$ and $w_1$ be positive weights. Suppose that $T$ is a sublinear operator that satisfies the estimates
$$
\| T f  \|_{L^{p_i}(w_i)} \le M_i \| f \|_{L^{p_i}(w_i)}, \quad i=1,2.
$$
Let $t \in (0,1)$ and define
$
1/p=(1-t)/p_0+t/p_1
$
and $w=w_0^{p(1-t)/p_0}w_1^{pt/p_1}$. Then $T$ satisfies the estimate
$$
\| T f   \|_{L^{p}(w)} \le M_0^{1-t}M_1^t \| f  \|_{L^{p}(w)}.
$$
\end{prop}
\begin{prop}\label{prop:Mlambda}
For all $p\in(1,\infty)$ and all $w\in A_p$, there are constants $C=C(p,w),\eta=\eta(p,w)>0$ such that
\begin{equation*}
    \Norm{M_{\calD_\lambda(\sigma)}f}{L^p(w)}\leq C\cdot D(\lambda)^{1-\eta}\Norm{f}{L^p(w)},
\end{equation*}
where $D(\lambda):=\max(\lambda,\lambda^{-1})$.
\end{prop}

\begin{proof}
The parameter $\sigma$ plays no role in this argument, so we drop it from the notation. Since $\calD_\lambda$ has the same nestedness structure as the usual $\calD_{\boxx}$, we have the unweighted bound
\begin{equation*}
  \Norm{M_{\calD_\lambda}f}{L^s}\leq s'\Norm{f}{L^s},\quad\forall s\in(1,\infty).
\end{equation*}
On the other hand, for any $I \in \calD_\lambda$, there is some $J\in \calD_{\boxx}$ such that $I\subset J$ and $|J|\le D(\lambda)|I|$. Therefore, we have
\[
M_{\calD_\lambda}f(x)= \sup_{I\in \calD_\lambda} \ave{|f|}_I 1_I (x)\le D(\lambda) \sup_{J\in \calD_{\boxx}} \ave{|f|}_J 1_J (x) = D(\lambda) M_{\calD_{\boxx}}f(x),
\]
and so
\begin{equation*}
  \Norm{M_{\calD_\lambda}f}{L^s(w)}\leq C(s,w)D(\lambda)\Norm{f}{L^s(w)},\quad\forall s\in(1,\infty),\quad\forall w\in A_s.
\end{equation*}

Let us now consider $s\in(1,\infty)$ and $w\in A_s$ fixed.  It is well known that we can find a $\delta=\delta(s,w)>0$ such that $w^{1+\delta}\in A_s$, and thus
\begin{equation*}
  \Norm{M_{\calD_\lambda}f}{L^s(w^{1+\delta})}\leq C(s,w^{1+\delta})D(\lambda) \Norm{f}{L^s(w^{1+\delta})}.
\end{equation*}
Now $w=(w^{1+\delta})^{1/(1+\delta)}\cdot 1^{\delta/(1+\delta)}$ and Proposition \ref{prop:SW} shows that
\begin{equation*}
  \Norm{M_{\calD_\lambda}f}{L^s(w)}\leq\big(C(s,w^{1+\delta})D(\lambda))^{1/(1+\delta)}(s')^{\delta/(1+\delta)}\Norm{f}{L^s(w)}.
\end{equation*}
Set $\eta:=\delta/(1+\delta)$. We have found $\eta=\eta(\delta)=\eta(s,w)>0$ such that
\begin{equation*}
  \Norm{M_{\calD_\lambda}f}{L^s(w)}\leq C(s,w)D(\lambda)^{1-\eta(s,w)}\Norm{f}{L^s(w)}.
\end{equation*}
\end{proof}

In addition to the usual Haar functions, we will need the functions $H_{I,J}$, where $I$ and $J$ are cubes with equal side length.
The functions $H_{I,J}$ satisfy
\begin{enumerate}
\item $H_{I,J}$ is supported on $I \cup J$ and constant on the children of $I$ and $J$,
\item $|H_{I,J}| \le |I|^{-1/2}$ and
\item $\int H_{I,J} = 0$.
\end{enumerate}
We denote (by slightly abusing notation) a general cancellative Haar function $h_{I}^{\eta}$, $\eta \ne (0, 0)$, simply by $h_I$.
\begin{defn}
  For $k = (k^1, k^2)$, $k^i \ge 0$, we define that the operator $Q_{k, \sigma}$ has either the form
  $$
  \langle Q_{k, \sigma}f, g\rangle = \sum_{K \in \calD_{2^{k^1-k^2}}(\sigma)}
  \sum_{\substack{I, J \in \calD_{\boxx}(\sigma) \\ I^{(k)} = J^{(k)} = K }} a_{IJK} \langle f, H_{I,J} \rangle \langle g, h_J \rangle
  $$
  or the symmetric form, and here $I^{(k)} = I^{(k^1)} \times I^{(k^2)}$ and the constants $ a_{IJK}$ satisfy
  $$
  |a_{IJK}| \le \frac{|I|}{|K|}.
  $$
\end{defn}
\begin{lem}\label{lem:aux}
  For $p \in (1, \infty)$ we have
  $$
  \| Q_{k, \sigma}f \|_{L^p} \lesssim (1+\max(k^1, k^2))^{1/2} \|f\|_{L^p}.
  $$ 
Moreover, for $w\in A_p$, there is $\eta>0$ such that
\begin{equation*}
   \Norm{Q_{k,\sigma}f}{L^p(w)}\lesssim (1+\max(k^1, k^2))^{1/2} 2^{\abs{k^1-k^2}(1-\eta)} \|f\|_{L^p(w)}.
\end{equation*}
\end{lem}

\begin{proof}
  We consider $\sigma$ fixed here and drop it from the notation.
  Suppose e.g. $k^1 \ge k^2$. We write
  $$
  \langle f, H_{I,J} \rangle = \big\langle E_{\frac{\ell(I)}{2}}f - E_{\ell(K^1)}f, H_{I,J} \big\rangle, \qquad   E_{\lambda} f := \sum_{\substack{L \in \calD_{\boxx} \\ \ell(L) = \lambda}} E_L f, \,
  E_L f = \langle f \rangle_L 1_L.
  $$
  Therefore, we have $\langle f, H_{I,J} \rangle = \langle \gamma_{K, k^1} f,  H_{I,J} \rangle$, where
  $$
  \gamma_{K, k^1} f := 1_K \sum_{\substack{L \in \calD_{\boxx} \\ 2^{-k^1}\ell(K^1) \le \ell(L) \le \ell(K^1)}} \Delta_L f,
  \qquad \Delta_L f= \sum_{\substack{L' \in \calD_{\boxx} \\ L' \subset L,\, \ell(L') = \frac{\ell(L)}{2}}}
   E_{L'} f - E_{L} f.
  $$
  Notice now that for $w \in A_2$ we have
  \begin{align*}
    \Big\| \Big( 
      \sum_{K \in \calD_{2^{k^1-k^2}}} |\gamma_{K, k^1} f|^2
      \Big)^{\frac{1}{2}}\Big\|_{L^2(w)}^2 
      &= \sum_{G \in \calD_{\boxx}}     \Big\| 
        \sum_{\substack{K \in \calD_{2^{k^1-k^2}} \\ K^{(0, k^1-k^2) = G}}} \gamma_{K, k^1} f
        \Big\|_{L^2(w)}^2  \\
        &= \sum_{G \in \calD_{\boxx}}     \Big\|
        \sum_{\substack{L \in \calD_{\boxx}, \, L \subset G \\  \ell(L) \ge 2^{-k^1}\ell(G)}} \Delta_L f
        \Big\|_{L^2(w)}^2 \\
        &\sim \sum_{G \in \calD_{\boxx}} \sum_{\substack{L \in \calD_{\boxx}, \, L \subset G \\  \ell(L) \ge 2^{-k^1}\ell(G)}}  \|\Delta_L f\|_{L^2(w)}^2
        \lesssim (1+k^1) \|f\|_{L^2(w)}^2,
  \end{align*}
  where we used the standard weighted square function estimate
  $$
  \sum_{L \in \calD_{\boxx}} \|\Delta_L f\|_{L^2(w)}^2 \sim \|f\|_{L^2(w)}^2
  $$
  twice in the end.

  To bound $Q_k f$ we need to estimate
  $$
  \sum_{K \in \calD_{2^{k^1-k^2}}} \frac{1}{|K|} \sum_{\substack{I, J \in \calD_{\boxx} \\ I^{(k)} = J^{(k)} = K }}
  \langle |\gamma_{K, k^1}f|, 1_{I} + 1_{J}\rangle \langle |\Delta_{K, k}g |, 1_J \rangle, \qquad
  \Delta_{K, k}g := \sum_{\substack{J \in \calD_{\boxx} \\ J^{(k^1, k^2)} = K }} \Delta_J g.
  $$
  We split this into two pieces according to $1_{I} + 1_{J}$. The first piece is bounded by
  \begin{equation*}
\begin{split}
  \sum_{K \in \calD_{2^{k^1-k^2}}} &\int \langle |\gamma_{K, k^1}f| \rangle_K |\Delta_{K, k}g |
  \le \sum_{K \in \calD_{2^{k^1-k^2}}}\int M_{\calD_{2^{k^1-k^2}}} \gamma_{K, k^1}f \cdot |\Delta_{K, k}g | \\
  &\leq \BNorm{\Big(\sum_{K\in\calD_{2^{k^1-k^2}}}[M_{\calD_{2^{k^1-k^2}}} \gamma_{K, k^1}f]^2\Big)^{1/2}}{L^2(w)}
  \BNorm{\Big(\sum_{K\in\calD_{2^{k^1-k^2}}} |\Delta_{K, k}g |^2\Big)^{1/2}}{L^2(w^{-1})} \\
  &\lesssim 2^{(k^1-k^2)(1-\eta)}\BNorm{\Big(\sum_{K\in\calD_{2^{k^1-k^2}}}\abs{\gamma_{K, k^1}f}^2\Big)^{1/2}}{L^2(w)}\Norm{g}{L^2(w^{-1})} \\
  &\lesssim 2^{(k^1-k^2)(1-\eta)}(1+k^1)^{1/2}\Norm{f}{L^2(w)}\Norm{g}{L^2(w^{-1})} \\
\end{split}
\end{equation*}
  while the second piece is bounded by
  $$
  \sum_{K \in \calD_{2^{k^1-k^2}}} \Big\langle \sum_{\substack{J \in \calD_{\boxx} \\ J^{(k^1, k^2)} = K }} \langle |\gamma_{K, k^1}f| \rangle_J 1_J, |\Delta_{K, k}g | \Big \rangle \le
  \sum_{K \in \calD_{2^{k^1-k^2}}} \int M_{\calD_{\boxx}} \gamma_{K, k^1}f \cdot |\Delta_{K, k}g |,
  $$
  which is estimated similarly except that the bound for $M_{\calD_{2^{k^1-k^2}}}$ is replaced by the standard result for $M_{\calD_{\boxx}}$.
This proves the claimed bounds in $L^2(w)$, and the results for $L^p(w)$ follow by Rubio de Francia's extrapolation theorem (the correct $1-\eta$ dependence
is maintained by the extrapolation, see Remark \ref{K:rem1}).

For the unweighted estimate in $L^p$ (with better complexity dependence), simply run the above argument
using the Fefferman--Stein $L^p(\ell^2)$ estimate for the strong maximal function instead of the $L^2(w)$ estimate of $M_{\calD_{2^{k^1-k^2}}}$,
and use the analogous $L^p(\ell^2)$ estimate for the square function involving $\gamma_{K, k^1}$ that follows
via Rubio de Francia extrapolation from the proved $L^2(w)$ estimate of the same square function.
\end{proof}

\begin{rem}\label{K:rem1}
  It is clear that  when $p=2$, $\eta$ depends only on the $A_2$ constant of $w$. In fact, in the proof of Proposition \ref{prop:Mlambda}
  we get $\eta \sim 1/{[w]_{A_2}}$. So we have 
  \[
   \Norm{Q_{k,\sigma}f}{L^2(w)}\le (1+\max(k^1, k^2))^{1/2} 2^{\abs{k^1-k^2}} N([w]_{A_2}) \|f\|_{L^2(w)}
  \]
   with
  \[
  N([w]_{A_2})=K([w]_{A_2}) 2^{- c\abs{k^1-k^2}/{[w]_{A_2}}},
  \]where $K$ is an increasing function. Hence $N$ is also an increasing function. Then standard extrapolation (see e.g. \cite[Theorem 3.1]{DU}) gives that the $L^p(w)$ bound of $Q_{k,\sigma}$ is 
  \[
  (1+\max(k^1, k^2))^{1/2} 2^{\abs{k^1-k^2}} N(c_p[w]_{A_p}^{\alpha(p)}).
  \]Then we get the desired estimate with $\eta=c(c_p[w]_{A_p}^{\alpha(p)})^{-1}$.
  \end{rem}

\begin{defn}
  We say that $\pi_b$ is a (one-parameter) paraproduct if it has the form
  $$
  \langle \pi_b f, g \rangle = \sum_{I \in \calD_{\boxx}} \langle b, h_I \rangle \langle f \rangle_I \langle g, h_I \rangle
  $$
  or the symmetric form.
\end{defn}
It is well-known (and follows readily from $H^1$--$\BMO$ duality) that paraproducts are $L^p$ bounded for $p\in(1,\infty)$ (and $L^p(w)$ bounded for $w\in A_p$) precisely when $b \in \BMO$.

\begin{thm}\label{thm:rep}
  Let $B$ be a $CZX(\R^2)$-form satisfying the weak boundedness property and the $T1$ conditions. Then we have
  $$
  B(f, g) =  \E_{\sigma}\Big[C_T \sum_{k^1, k^2 \ge 0}  2^{-\theta_2(k_{\max}-k_{\min})}2^{-\theta_1 k_{\min}} 
  \langle Q_{k,\sigma}f, g \rangle
  + \langle \pi_{b_1, \sigma} f, g\rangle + \langle \pi_{b_2, \sigma} f, g\rangle\Big],
  $$
 where $k_{\max}=\max_{i=1,2}k^i$, $k_{\min}=\min_{i=1,2}k^i$.
  In particular, we have for $p \in (1, \infty)$ that
  $$
  |B(f, g)| \lesssim \|f\|_{L^p} \|g\|_{L^{p'}}.
  $$
  If $\theta_2=1$, we also have for all $w\in A_p$ that
  \begin{equation*}
  \abs{B(f,g)}\lesssim \Norm{f}{L^p(w)}\Norm{g}{L^{p'}(w^{1-p'})}.
\end{equation*}
\end{thm}

\begin{proof}
  Write (by expanding $f = \sum_I \Delta_I f$, $g = \sum_J \Delta_J g$ and collapsing the off-diagonal)
  $$
  B(f,g) = \E_{\sigma}\sum_{\ell(I) = \ell(J)} [B(E_I f, \Delta_J g) + B(\Delta_I f, E_J g) + B(\Delta_I f, \Delta_J g)],
  $$
  where $I, J \in \calD_{\boxx}(\sigma)$.
  We begin by writing 
  \begin{align*}
    \Sigma_1 &:= \E_{\sigma}\sum_{\ell(I) = \ell(J)} B(E_I f, \Delta_J g) \\
    &= \E_{\sigma} \sum_{\ell(I) = \ell(J)}  B(h_I^0, h_J) \langle f, H_{I,J} \rangle \langle g, h_J \rangle +
    \E_{\sigma} \sum_{J \in \calD_{\boxx}(\sigma)} B(1, h_J) \langle f \rangle_J \langle g, h_J \rangle =: \Sigma_{1, 1} +
    \Sigma_{1,2},
  \end{align*}
  where $H_{I,J} := h_I^0 - h_J^0$. As the term $\Sigma_{1,2}$ is readily a paraproduct, we only continue with $\Sigma_{1, 1}$.
  This was a standard one-parameter start. Write
  $$
  \Sigma_{1, 1} = \E_{\sigma}\sum_{m = (m^1, m^2) \in \Z^2 \setminus \{0\}} \sum_{I \in \calD_{\boxx}(\sigma)} \varphi_{I, I \dotplus m},
  \qquad \varphi_{I, I \dotplus m} = B(h_I^0, h_{I \dotplus m}) \langle f, H_{I,I \dotplus m} \rangle \langle g, h_{I \dotplus m} \rangle,
  $$
  where $I \dot{+} m:= I + m\ell(I) \in \calD_{\boxx}(\sigma)$. Next, write
  $$
  \sum_{m = (m^1, m^2) \in \Z^2 \setminus \{0\}} = \sum_{m^1 \in \Z \setminus \{0\}}\sum_{m^2 \in \Z \setminus \{0\}}
  + \sum_{\substack{m^2 \in \Z \setminus \{0\} \\ m = (0, m^2)} }  + \sum_{\substack{m^1 \in \Z \setminus \{0\} \\ m = (m^1, 0)} }.
  $$

  Focusing, for now, on the part $m^1 \ne 0$ and $m^2 \ne 0$, write
  $$
  \sum_{m^1 \in \Z \setminus \{0\}}\sum_{m^2 \in \Z \setminus \{0\}} = \sum_{k^1 = 2}^{\infty} \sum_{k^2 = 2}^{\infty}
  \sum_{|m^1| \in (2^{k^1-3}, 2^{k^1-2}]} \sum_{|m^2| \in (2^{k^2-3}, 2^{k^2-2}]}. 
  $$
  We have by independence and \eqref{eq1} that
  \begin{align*}
    \E_{\sigma} &\sum_{k^1 = 2}^{\infty} \sum_{k^2 = 2}^{\infty} 
    \sum_{|m^1| \in (2^{k^1-3}, 2^{k^1-2}]} \sum_{|m^2| \in (2^{k^2-3}, 2^{k^2-2}]} \sum_{I \in \calD_{\boxx}(\sigma)} \varphi_{I, I \dotplus m} \\
    &= 4 \E_{\sigma} \sum_{k^1 = 2}^{\infty} \sum_{k^2 = 2}^{\infty} 
    \sum_{|m^1| \in (2^{k^1-3}, 2^{k^1-2}]} \sum_{|m^2| \in (2^{k^2-3}, 2^{k^2-2}]} \sum_{I \in \calD_{\boxx}(\sigma, k)} \varphi_{I, I \dotplus m},
  \end{align*}
  where $k = (k^1, k^2)$, and the gist is that for $|m^1| \in (2^{k^1-3}, 2^{k^1-2}]$, $|m^2| \in (2^{k^2-3}, 2^{k^2-2}]$ and
  $I \in \calD_{\boxx}(\sigma, k)$ we have
  $$
  (I^1 \dotplus m^1)^{(k^1)} = (I^1)^{(k^1)} =: K^1 \qquad \textup{and} \qquad
  (I^2 \dotplus m^2)^{(k^2)} = (I^2)^{(k^2)} =: K^2.
  $$
  We have that $K = I^{(k)} = (I\dotplus m)^{(k)} = K^1 \times K^2  \in \calD_{2^{k^1 - k^2}}(\sigma)$,
  since
  $$
  \ell(K^1) =  2^{k^1} \ell(I^1) =  2^{k^1} \ell(I^2) =  2^{k^1-k^2} 2^{k^2} \ell(I^2) = 2^{k^1-k^2} \ell(K^2).
  $$
  Finally, notice that by Lemma \ref{lem:kernelest} we have
  $$
  |B(h_I^0, h_{I \dotplus m})| \lesssim 2^{-k^1}2^{-k^2}\frac{2^{k_{\min}(\theta_2-\theta_1)}}{2^{k_{\max}\theta_2}}  
   =2^{-\theta_2(k_{\max}-k_{\min})}2^{-\theta_1 k_{\min}}\frac{\abs{I}}{\abs{K}}. 
  $$
  The sums, where $m^1 = 0$ or $m^2 = 0$ are completely similar (just do the above in one of the parameters). We are done with $\Sigma_1$.

  Of course, $\Sigma_2 := \E_{\sigma}\sum_{\ell(I) = \ell(J)} B(\Delta_I f, E_J g)$ is completely symmetric.
  The term $\Sigma_3 := \E_{\sigma}\sum_{\ell(I) = \ell(J)} B(\Delta_I f, \Delta_J g)$ does not produce a paraproduct
  and produces shifts with the simpler form $H_{I, J} = h_I$.

  The unweighted boundedness follows immediately from the $L^p$ bounds of the paraproducts and the bound $\Norm{Q_{k,\sigma}f}{L^p}\lesssim(1+k_{\max})^{1/2}\Norm{f}{L^p}$, since the exponentially decaying factor $2^{-\theta_2(k_{\max}-k_{\min})}2^{-\theta_1 k_{\min}}$
   clearly make the series summable for any $\theta_1,\theta_2>0$.
  
  Let us finally consider the weighted case with $\theta_2=1$. Then for some $\eta=\eta(p,w)>0$, we have
  \begin{equation*}
  \begin{split}
   &2^{-\theta_2(k_{\max}-k_{\min})}2^{-\theta_1 k_{\min}}\abs{\pair{Q_{k,\sigma}f}{g}} \\
  &\lesssim 2^{-(k_{\max}-k_{\min})}2^{-\theta_1 k_{\min}} 2^{(k_{\max}-k_{\min})(1-\eta)}(1+k_{\max})^{1/2}\Norm{f}{L^p(w)}\Norm{g}{L^{p'}(w^{1-p'})} \\
  &= 2^{-\eta(k_{\max}-k_{\min})}2^{-\theta_1 k_{\min}} (1+k_{\max})^{1/2}\Norm{f}{L^p(w)}\Norm{g}{L^{p'}(w^{1-p'})},
\end{split}
\end{equation*}
and again we have exponential decay that makes the series over $k^1,k^2$ summable.  
\end{proof}
\begin{rem}\label{rem:log}
  If $\theta_2 = 1$ we may redefine $D_1(x, y)$ to be the slightly larger quantity
  $$
  D_1(x, y) := \Bigg(\frac{|x^1-y^1|}{|x^2-y^2|} + \frac{|x^2-y^2|}{|x^1-y^1|}\Bigg)^{-1}
  \log\Bigg( \frac{|x^1-y^1|}{|x^2-y^2|} + \frac{|x^2-y^2|}{|x^1-y^1|} \Bigg),
  $$
  and still prove the weighted estimates essentially like above. This is pertinent in the sense that
  if we take a Fefferman-Pipher multiplier \cite{FP} -- a singular integral of Zygmund type --
  and use it to induce a CZX operator, a logarithmic term appears. In this threshold a weighted
  estimate still holds. See also \cite{HLMV:ZYG}.
\end{rem}

\section{Commutator estimates}
We show that our exotic Calder\'on-Zygmund operators also satisfy the usual one-parameter commutator estimates.
Since weighted estimates with one-parameter weights do not in general hold (see Section \ref{sec:weighted}), this does not
follow from the well-known Cauchy integral trick.
\begin{thm}
Let $T\in\bddlin(L^2(\R^2))$ be an operator associated with a CZX kernel $K$. Then we have
 $$
 \| [b,T]f \|_{L^p} \lesssim \|b\|_{\BMO} \| f \|_{L^p}
 $$
 whenever $p \in (1, \infty)$. Here $[b,T]f:=bTf-T(bf)$.
\end{thm}
\begin{proof}
By Theorem \ref{thm:rep} and the well-known commutator estimates for the paraproducts $\pi$, we only need to prove that
\[
\big|\langle Q_{k,\sigma}(bf), g \rangle-\langle Q_{k,\sigma}f, bg \rangle\big|\lesssim \varphi(k)\|b\|_{\BMO}\|f\|_{L^p}\|g\|_{L^{p'}},
\]where $\varphi$ is some polynomial.
We consider $\sigma$ fixed and drop it from the notation.
Recall the usual paraproduct decomposition of $bf$:
\[
bf=a_1(b,f)+a_2(b,f)+a_3(b,f),
\]where 
\[
a_1(b,f)=\sum_{I \in \calD_{\boxx}} \Delta_I b \Delta_I f,\quad a_2(b,f)=\sum_{I \in \calD_{\boxx}} \Delta_I b \langle f\rangle_I, \quad  a_3(b,f)= \sum_{I \in \calD_{\boxx}} \langle b\rangle_I \Delta_I f.
\]
Invoking the above decomposition, the well-known boundedness of paraproducts 
$$
\| a_i(b, f) \|_{L^p} \lesssim \|b\|_{\BMO} \| f \|_{L^p}, \qquad i = 1,2,
$$
and Lemma \ref{lem:aux}, it suffices to control 
\begin{equation}\label{eq2}
  \begin{split}
    &\big|\langle Q_{k}(a_3(b,f)), g \rangle-\langle Q_{k}f, a_3(b,g) \rangle\big|\\
    &=  \Big| \sum_{K \in \calD_{2^{k^1-k^2}}}
    \sum_{\substack{I, J \in \calD_{\boxx} \\ I^{(k)} = J^{(k)} = K }} a_{IJK}\Big[ \langle a_3(b,f), H_{I,J} \rangle \langle g, h_J \rangle- \langle b\rangle_J\langle f, H_{I,J} \rangle \langle g, h_J \rangle\Big] \Big|.
\end{split}
\end{equation}
We may assume $k^1\ge k^2$.
There holds that 
\begin{align*}
&\langle a_3(b,f), H_{I,J} \rangle\langle g, h_J \rangle- \langle b\rangle_J\langle f, H_{I,J} \rangle \langle g, h_J \rangle\\&\hspace{3cm}= \sum_{\substack{Q\in \calD_{\boxx}, Q\subset  K^{(0, k^1-k^2)}\\ \ell(I)\le \ell(Q)\le 2^{k^1}\ell(I)}} (\langle b\rangle_Q-\langle b\rangle_J)  \langle \Delta_Q f, H_{I,J}\rangle\langle g, h_J \rangle.
\end{align*}Observe that $|\langle b\rangle_Q-\langle b\rangle_J|\lesssim k^1\|b\|_{\BMO}$. On the other hand, 
since we only need to consider those $Q$ such that $\langle \Delta_Q f, H_{I,J}\rangle\neq 0$, i.e., either $I\subset Q$ or $J\subset Q$, we have  
\begin{align*}
& \sum_{\substack{Q\in \calD_{\boxx}, Q\subset  K^{(0, k^1-k^2)}\\ \ell(I)\le \ell(Q)\le 2^{k^1}\ell(I)}} \big| (\langle b\rangle_Q-\langle b\rangle_J) \langle \Delta_Q f, H_{I,J}\rangle\langle g, h_J \rangle\big|\\
&\hspace{2cm}\lesssim k^1\|b\|_{\BMO} \sum_{\ell^1=0}^{k^1}\Big(\big|\langle \Delta_{I^{(\ell^1)}} f, H_{I,J}\rangle\langle g, h_J \rangle\big|+ \big| \langle \Delta_{J^{(\ell^1)}} f, H_{I,J}\rangle\langle g, h_J \rangle\big|\Big)\\
&\hspace{2cm}\le 2k^1\|b\|_{\BMO} \sum_{\ell^1=0}^{k^1}\Big(\langle | \Delta_{I^{(\ell^1)}} f |, h_{I}^0\rangle|\langle g, h_J \rangle|+ 
 \langle | \Delta_{J^{(\ell^1)}} f|, h_J^0\rangle |\langle g, h_J \rangle|\Big),
\end{align*}
where we have used the simple observation
\[
\langle |\Delta_{I^{(\ell^1)}} f|, h_{J}^0\rangle 
\le \langle |\Delta_{J^{(\ell^1)}} f|, h_{J}^0\rangle, 
\quad \langle |\Delta_{J^{(\ell^1)}} f|, h_{I}^0\rangle 
\le \langle |\Delta_{I^{(\ell^1)}} f|, h_{I}^0\rangle.
\]

Now, returning to \eqref{eq2}, for a fixed $\ell^1$, we have 
\begin{align*}
& \sum_{K \in \calD_{2^{k^1-k^2}}}
  \sum_{\substack{I, J \in \calD_{\boxx} \\ I^{(k)} = J^{(k)} = K }} \frac{|I|}{|K|}\langle |\Delta_{I^{(\ell^1)}} f|, h_{I}^0\rangle |\langle g, h_J \rangle|\\
  &\hspace{3cm}\le \sum_{K \in \calD_{2^{k^1-k^2}}}\Big\langle \sum_{\substack{I \in \calD_{\boxx} \\ I^{(k)}  = K }}| \Delta_{I^{(\ell^1)}} f|1_I, M_{\calD_{2^{k^1-k^2}}}|\Delta_{K,k} g|\Big\rangle.
\end{align*}
Using extrapolation we only need to show that 
\begin{equation*}
\sum_{K \in \calD_{2^{k^1-k^2}}} \Big\| \sum_{\substack{I \in \calD_{\boxx} \\ I^{(k)}  = K }}| \Delta_{I^{(\ell^1)}} f|1_I\Big\|_{L^2(w)}^2\lesssim \|f\|_{L^2(w)}^2.
\end{equation*}
However, this is clear because
\begin{align*}
\sum_{K \in \calD_{2^{k^1-k^2}}} \Big\| \sum_{\substack{I \in \calD_{\boxx} \\ I^{(k)}  = K }}| \Delta_{I^{(\ell^1)}} f|1_I\Big\|_{L^2(w)}^2&=  \sum_{G \in \calD_{\boxx}} \sum_{\substack{K \in \calD_{2^{k^1-k^2}}\\ K^{(0, k^1-k^2)}=G}} \sum_{\substack{I \in \calD_{\boxx} \\ I^{(k)}  = K }} \big\| 1_I \Delta_{I^{(\ell^1)}} f \big\|_{L^2(w)}^2\\
&=  \sum_{G \in \calD_{\boxx}}\sum_{\substack{I \in \calD_{\boxx} \\ I^{(k^1, k^1)}  = G }}\big\|  1_I\Delta_{I^{(\ell^1)}} f\big\|_{L^2(w)}^2\\
&= \sum_{I \in \calD_{\boxx}} \big\|  1_I\Delta_{I^{(\ell^1)}} f\big\|_{L^2(w)}^2
=\sum_{Q\in \calD_{\boxx}}\big\|  \Delta_{Q} f\big\|_{L^2(w)}^2.
\end{align*}
To conclude the proof of the proposition we are left to deal with
\begin{align*}
& \sum_{K \in \calD_{2^{k^1-k^2}}}
  \sum_{\substack{I, J \in \calD_{\boxx} \\ I^{(k)} = J^{(k)} = K }} \frac{|I|}{|K|} \langle| \Delta_{J^{(\ell^1)}} f|, h_{J}^0\rangle |\langle g, h_J \rangle|\\
  &\hspace{3cm}\le \sum_{K \in \calD_{2^{k^1-k^2}}}\Big\langle \sum_{\substack{J\in \calD_{\boxx} \\ J^{(k)}  = K }}| \Delta_{J^{(\ell^1)}} f|1_J, M_{\calD_{\boxx}}|\Delta_{K,k} g|\Big\rangle.
\end{align*}
After the estimate above, this is clearly similar as the other term. We are done.
\end{proof}

\section{Counterexample to weighted estimates and sparse bounds}\label{sec:weighted}
We begin by showing that bounded $CZX(\R^2)$ operators need not be bounded with respect to the
one-parameter weights if $\theta_2 < 1$. 
\begin{lem}
  For scalars $\theta_2 \in (0, 1]$, $t_1, t_2 > 0$ and a bump function $\varphi$ we define
\begin{equation*}
  K(x) = K_{t_1, t_2, \theta_2}(x) = \Big(\frac{t_1}{t_2}+\frac{t_2}{t_1}\Big)^{-\theta_2}\prod_{i=1}^2\frac{1}{t_i}\phi(\frac{x_i}{t_i}).
\end{equation*}
Then, uniformly on $t_1, t_2$, $K\in CZX(\R^2)$ with $\theta_1=1$ and $\theta_2$ appearing in the very definition of $K$, and
$$
\|K * f\|_{L^2} \lesssim \|f\|_{L^2}.
$$
\end{lem}
\begin{proof}
  Suppose by symmetry that $t_1\leq t_2$. Then, using the rapid decay of all the derivatives of $\phi$,
  we have for all $N$ that
  \begin{equation}\label{eq3}
    \begin{split}
      &\abs{x_1^{1+\alpha_1}x_2^{1+\alpha_2}}\abs{\partial^\alpha K(x)} 
      \lesssim \Big(\frac{t_1}{t_2}\Big)^{\theta_2}\prod_{i=1}^2\Big(\frac{\abs{x_i}}{t_i}\Big)^{1+\alpha_i}\Big(1+\frac{\abs{x_i}}{t_i}\Big)^{-N}  \\
      &= \Big(\frac{t_1}{t_2}\Big)^{\theta_2}\prod_{i=1}^2\frac{\abs{x_i}^{1+\alpha_i}t_i^{N-1-\alpha_i}}{(\abs{x_i}+t_i)^N} \\
      &=\frac{\abs{x_1}^{1+\alpha_1-\theta_2}t_1^{N-\alpha_1-1+\theta_2}}{(\abs{x_1}+t_1)^N}\frac{\abs{x_2}^{1+\alpha_2+\theta_2}t_2^{N-\alpha_2-1-\theta_2}}{(\abs{x_2}+t_2)^N}\Big(\frac{\abs{x_1}}{\abs{x_2}}\Big)^{\theta_2} \\
      &=\Big(\frac{t_1}{t_2}\Big)^{2\theta_2}\frac{\abs{x_1}^{1+\alpha_1+\theta_2}t_1^{N-\alpha_1-1-\theta_2}}{(\abs{x_1}+t_1)^N}\frac{\abs{x_2}^{1+\alpha_1-\theta_2}t_2^{N-\alpha_1-1+\theta_2}}{(\abs{x_2}+t_2)^N}\Big(\frac{\abs{x_2}}{\abs{x_1}}\Big)^{\theta_2}.
  \end{split}
  \end{equation}
  In the last two lines of \eqref{eq3}, if $N$ is large enough, each factor in front of the last one can be bounded by one.
  Thus, we have
  \begin{equation*}
    \abs{x_1^{1+\alpha_1}x_2^{1+\alpha_2}}\abs{\partial^\alpha K(x)} \lesssim \min\Big\{ \Big(\frac{\abs{x_1}}{\abs{x_2}}\Big)^{\theta_2}, \Big(\frac{\abs{x_2}}{\abs{x_1}}\Big)^{\theta_2}\Big\}
    \sim\Big(\frac{\abs{x_1}}{\abs{x_2}}+\frac{\abs{x_2}}{\abs{x_1}}\Big)^{-\theta_2}.
  \end{equation*}
  From $\alpha\in\{(0,0),(0,1),(1,0)\}$, we get the desired kernel estimates.

  For the boundedness, notice that
  $$
  |\widehat K(\xi)| = \Big(\frac{t_1}{t_2}+\frac{t_2}{t_1}\Big)^{-\theta_2}\prod_{i=1}^2 |\widehat \phi(t_i \xi_i)|
  \le \prod_{i=1}^2 |\widehat \phi(t_i \xi_i)| \lesssim 1.
  $$
\end{proof}

We fix $t_1, t_2, \theta_2$ momentarily and denote $K = K_{t_1, t_2, \theta_2}$.
For any rectangle $R$ of sidelengths $t_1,t_2$, it is clear that $$
K*f\gtrsim \operatorname{ecc}(R)^{-\theta_2}1_R\ave{f}_R
$$
whenever $f\geq 0$ and
\begin{equation*}
   \operatorname{ecc}(R) :=\max\Big\{\frac{t_1}{t_2},\frac{t_2}{t_1}\Big\}
\end{equation*}
is the eccentricity of $R$.
Suppose now that for $p \in (1, \infty)$ we have $\Norm{K*f}{L^p(w)}\leq C([w]_{A_p}) N\Norm{f}{L^p(w)}$
for all $w \in A_p$ and $f \in L^p(w)$.
Then
\begin{equation*}
  \operatorname{ecc}(R)^{-\theta_2} w(R)^{1/p}\ave{f}_R =  \Norm{\operatorname{ecc}(R)^{-\theta_2} 1_R\ave{f}_R}{L^p(w)}\lesssim C([w]_{A_p}) N\Norm{f}{L^p(w)}.
\end{equation*}
If $f=1_R\sigma$, where $\sigma=w^{-1/(p-1)}$, then $f^p w=1_R\sigma$, and hence
\begin{equation*}
  \operatorname{ecc}(R)^{-\theta_2} w(R)^{1/p}\ave{\sigma}_R \lesssim C([w]_{A_p}) N\sigma(R)^{1/p}= C([w]_{A_p})  N\ave{\sigma}_R^{1/p}\abs{R}^{1/p},
\end{equation*}
or
\begin{equation*}
  \ave{w}_R^{1/p}\ave{\sigma}_R^{1/p'}\lesssim C([w]_{A_p})N\operatorname{ecc}(R)^{\theta_2}.
\end{equation*}

If all $L^2$ bounded $CZX$ operators would satisfy the $L^p(w)$ boundedness for all $w \in A_p$
with a bound $C([w]_{A_p}) N$, where $N$ depends only on the kernel constants and boundedness
constants of the operator, then for \textbf{all} rectangles $R\subset\R^2$ we have
\begin{equation}\label{eq4}
  \ave{w}_R^{p'/p}\ave{\sigma}_R\lesssim C([w]_{A_p}) \operatorname{ecc}(R)^{\theta_2p'}.
\end{equation}
This is because for the kernels $K = K_{t_1, t_2, \theta_2}$ the constant $N$ is uniformly bounded
on $t_1, t_2$.

Now consider $w(x)=\abs{x}^\alpha$, which belongs to $A_p(\R^2)$ if $-2<\alpha<2(p-1)$. Fix some $\alpha\in(p-1,2(p-1))$ for now.
We let our implicit constants depend on $\alpha$ as it is not important for our argument.
We consider a rectangle $R$ of the form $(0,\epsilon)\times(\epsilon,1)$ with eccentricity $\sim 1/\epsilon$. On $R$, we have $\abs{x}\sim x_2$, and then
\begin{equation*}
  \ave{w}_R\sim\epsilon^{-1}\int_0^\epsilon\Big(\int_\epsilon^1 x_2^{ \alpha }\ud x_2\Big)\ud x_1
  \sim\epsilon^{-1}\int_0^\epsilon 1\ud x_1 = 1
\end{equation*}
and
\begin{equation*}
  \ave{\sigma}_R
  \sim\epsilon^{-1}\int_0^\epsilon\Big(\int_\epsilon^1 x_2^{-\frac\alpha{p-1}}\ud x_2\Big)\ud x_1 
   \sim\epsilon^{-1}\int_0^\epsilon \epsilon^{1-\frac\alpha{p-1}}\ud x_1 = \epsilon^{1-\frac\alpha{p-1}}\sim\operatorname{ecc}(R)^{\frac\alpha{p-1}-1}.
\end{equation*}
 If \eqref{eq4} holds, then
 \begin{equation*}
   \operatorname{ecc}(R)^{\frac\alpha{p-1}-1}\lesssim \operatorname{ecc}(R)^{p'\theta_2}.
\end{equation*}
Since we can let $\operatorname{ecc}(R)\to\infty$, it means that we must have $\alpha-(p-1) - p\theta_2\leq 0$. Letting $\alpha \to 2(p-1)$ in this inequality gives $\theta_2 \ge 1/{p'}$. Then $\theta_2\ge 1$ follows by letting $p\to \infty$. 

Thus, weighted boundedness cannot hold in general for $CZX$ operators if $\theta_2 < 1$. Next, we prove some complementary sparse estimates, which
refine the weighted estimates in the case $\theta_2 = 1$.

\begin{defn}
We say that an operator $T$ satisfies pointwise $L^p$-sparse domination with constants $C$ and $\epsilon$, if for every compactly supported $f\in L^p(\R^d)$, there exists an $\epsilon$-sparse collection $\mathscr S$ of cubes such that
\begin{equation*}
  \abs{Tf}\leq C\sum_{S\in\mathscr S}\ave{f}_{S,p}1_S,
\end{equation*}
where $\ave{f}_{S,p}:=\ave{\abs{f}^p}_S^{1/p}$, and a collection $\mathscr S$ of cubes is called $\epsilon$-sparse, if there are pairwise disjoint sets $E(S)\subset S$ for every $S\in\mathscr S$ with $\abs{E(S)}\geq\epsilon\abs{S}$.
\end{defn}

There are by now several approaches to proving sparse domination. We will use one by Lerner and Ombrosi \cite{LO}, which depends on bounds on the following maximal function related to the operator $T$ under investigation
\begin{equation*}
  \mathcal M_{T,3}^{\#}f(x):=\sup_{J\owns x}\esssup_{y,z\in J}\abs{T(1_{(3J)^c}f)(y)-T(1_{(3J)^c}f)(z)},
\end{equation*}
where the supremum, once again, is over all cubes $J$.

\begin{lem}\label{lem:LOmax}
Let $T$ be an operator with a CZX kernel satisfying $\theta_2 = 1$. Then
\begin{equation*}
  \mathcal M_{T,3}^{\#}f(x)\lesssim M_* f(x):=\sup_{R\owns x}\fint_R\abs{f(y)}\ud y,
\end{equation*}
where the right-hand side is the strong maximal function, with supremum over all (axes-parallel) rectangles $R\subset\R^2$ containing $x$.
\end{lem}

\begin{proof}
Let us fix a cube $J\subset\R^2$ with centre $c_J$, and some $x,y,z\in J$. Note that
\begin{equation*}
\begin{split}
  T(1_{(3J)^c}f)(y) &-T(1_{(3J)^c}f)(z) \\
  &=[T(1_{(3J)^c}f)(y)-T(1_{(3J)^c}f)(c_J)]-[T(1_{(3J)^c}f)(z)-T(1_{(3J)^c}f)(c_J)]
\end{split}
\end{equation*}
and
\begin{equation*}
\begin{split}
  T(1_{(3J)^c}f)(y)-T(1_{(3J)^c}f)(c_J)
  =\int_{(3J)^c}[K(y,u)-K(c_J,u)]f(u)\ud u.
\end{split}
\end{equation*}
As usual, we split
\begin{equation*}
  (3J)^c=[(3J^1)^c\times 3J^2]\cup[3J^1\times (3J^2)^c]\cup[(3J^1)^c\times (3J^2)^c].
\end{equation*}
For the integral over the last component, we have from Lemma \ref{lem:CZXaway} that
\begin{equation*}
\begin{split}
  \int_{(3J^1)^c\times (3J^2)^c}& \abs{[K(y,u)-K(c_J,u)]f(u)}\ud u\\
  &\lesssim\int_{(3J^1)^c\times (3J^2)^c}\prod_{i=1}^2\frac{\ell(J)^\theta}{\dist(u^i,J^i)^{1+\theta}}\abs{f(u)}\ud u \\
  &\lesssim\sum_{k^1,k^2=1}^\infty\int_{(3^{k^1+1}J^1\setminus 3^{k^1}J^1)\times (3^{k^2+1}J^2\setminus 3^{k^2}J^2)}\prod_{i=1}^2\frac{\ell(J)^\theta}{(3^{k^i}\ell(J))^{1+\theta}}\abs{f(u)}\ud u \\
  &\lesssim\sum_{k^1,k^2=1}^\infty 3^{-(k^1+k^2)\theta}\fint_{3^{k^1+1}J^1\times 3^{k^2+1}J^2}\abs{f(u)}\ud u \\
  &\lesssim\sum_{k^1,k^2=1}^\infty 3^{-(k^1+k^2)\theta}M_*f(x) \lesssim M_*f(x).
\end{split}
\end{equation*}
For the other components, using that $\theta_2 = 1$ we have directly from the size estimate that
\begin{equation*}
\begin{split}
  &\int_{(3J^1)^c\times 3J^2} \abs{K(y,u)f(u)}\ud u\\
  &\lesssim \int_{(3J^1)^c\times 3J^2}\frac{1}{\abs{y^1-u^1}^2}\abs{f(u)}\ud u \\
  &\lesssim\sum_{k^1=1}^\infty\int_{(3^{k^1+1}J^1\setminus 3^{k^1}J^1)\times 3J^2}\frac{1}{(3^{k^1}\ell(J))^{2}}\abs{f(u)}\ud u \\
  &\lesssim\sum_{k^1=1}^\infty 3^{-k^1}\fint_{3^{k^1+1}J^1\times 3J^2}\abs{f(u)}\ud u \\
  &\lesssim\sum_{k^1=1}^\infty 3^{-k^1}M_*f(x) \lesssim M_*f(x).
\end{split}
\end{equation*}
Altogether, we have checked that
\begin{equation*}
  \abs{T(1_{(3J)^c}f)(y) -T(1_{(3J)^c}f)(z)}\lesssim M_*f(x),
\end{equation*}
and taking the supremum over $y,z\in J$ and then over $J\owns x$ we see that
\begin{equation*}
  \mathcal M_{T,3}^{\#}f(x)\lesssim M_* f(x).
\end{equation*}
\end{proof}

We now quote a slight variant of a result of Lerner and Ombrosi \cite[Theorem 1.1]{LO}:

\begin{thm}[Lerner and Ombrosi \cite{LO}]\label{thm:LO}
Let $T$ be a sublinear operator that is bounded from $L^q(\R^d)$ to $L^{q,\infty}(\R^d)$, and such that $\mathcal M_{T,3}^{\#}$ is bounded from $L^r(\R^d)$ to $L^{r,\infty}(\R^d)$ for some $1\leq q,r<\infty$. Let $s=\max(q,r)$. Then $T$ satisfies pointwise $L^s$-sparse domination with constants
\begin{equation*}
   C=c_{d}(\Norm{T}{L^q\to L^{q,\infty}}+\Norm{\mathcal M_{T,3}^{\#}}{L^r\to L^{r,\infty}}),\qquad \epsilon=\epsilon_d.
\end{equation*}
\end{thm}

\begin{proof}
This is essentially \cite[Theorem 1.1]{LO}, except for some details mainly related to the constant $C$. Since this constant will be relevant to us below, we will explain the necessary changes.
On a more trivial side, the statement in \cite{LO} involves an additional parameter $\alpha$ in the maximal operator $\mathcal M_{T,\alpha}^{\#}$; we simply take $\alpha=3$. Also, in \cite{LO} the $L^q(\R^d)$-to-$L^{q,\infty}(\R^d)$ boundedness is replaced by a certain ``$W_q$ condition''; however, this follows from the $L^q(\R^d)$-to-$L^{q,\infty}(\R^d)$ boundedness, as pointed out shortly before \cite[Theorem 1.1]{LO}.

More seriously, the bound for $C$ given in \cite{LO} has dependencies on additional parameters that we wish to avoid. For this, it is necessary to inspect the proof of \cite[Theorem 1.1.]{LO}. The said proof provides the expression
\begin{equation*}
  C=(3+c)A,
\end{equation*}
where $c$ and $A$ need to be chosen so that each of the sets
\begin{equation*}
  \{M(\abs{f}^s1_{3Q})^{1/s}>c\ave{|f|}_{3Q,s}\},\quad\{\abs{T(f1_{3Q})}>A\ave{|f|}_{3Q,s}\},\quad\{\mathcal M_{T,3}^{\#}(f1_{3Q})>A\ave{|f|}_{3Q,s}\}
\end{equation*}
have measure at most $\epsilon_d\abs{Q}$ for some small dimensional $\epsilon_d$. However, we have
\begin{equation*}
  \abs{\{M(\abs{f}^s1_{3Q})>(c\ave{|f|}_{3Q,s})^s\}}
  \leq \frac{C_d}{(c\ave{|f|}_{3Q,s})^s}\Norm{\abs{f}^s 1_{3Q}}{1}
  =\frac{C_d}{c^s}\abs{3Q}=\frac{C_d 3^d}{c^s}\abs{Q},
\end{equation*}
where $C_d=\Norm{M}{L^1(\R^d)\to L^{1,\infty}(\R^d)}\geq 1$, so that we can take
\begin{equation*}
  c=\Big(\frac{C_d 3^d}{\epsilon_d}\Big)\geq \Big(\frac{C_d 3^d}{\epsilon_d}\Big)^{1/s},
\end{equation*}
since $s\geq 1$. Thus $c=c_d$. On the other hand, using among other things that $q\leq s$ and H\"older's inequality, we have
\begin{equation*}
\begin{split}
  \abs{\{\abs{T(f1_{3Q})}>A\ave{|f|}_{3Q,s}\}}
  &\leq\frac{\Norm{T}{L^q\to L^{q,\infty}}^q}{(A\ave{|f|}_{3Q,s})^q}\Norm{f 1_{3Q}}{L^q}^q
  \leq\frac{\Norm{T}{L^q\to L^{q,\infty}}^q}{(A\ave{|f|}_{3Q,q})^q}\Norm{f 1_{3Q}}{L^q}^q \\
  &=\frac{\Norm{T}{L^q\to L^{q,\infty}}^q}{A^q}\abs{3Q}=\frac{3^d\Norm{T}{L^q\to L^{q,\infty}}^q}{A^q}\abs{Q},
\end{split}
\end{equation*}
so to make this at most $\epsilon_d$, it is enough to take
\begin{equation*}
  A\geq\frac{3^d}{\epsilon_d}\Norm{T}{L^q\to L^{q,\infty}}\geq \Big(\frac{3^d}{\epsilon_d}\Big)^{1/q}\Norm{T}{L^q\to L^{q,\infty}},
\end{equation*}
since $q\geq 1$. Similarly, with $\mathcal M_{T,3}^{\#}$ in place of $T$ and $r$ in place of $q$ in order that $\abs{\{\mathcal M_{T,3}^{\#}>A\ave{|f|}_{3Q,s}\}}\leq\epsilon_d\abs{Q}$, it is enough to take
\begin{equation*}
  A\geq\frac{3^d}{\epsilon_d}\Norm{\mathcal M_{T,3}^{\#}}{L^r\to L^{r,\infty}}.
\end{equation*}
So an admissible choice is $c=c_d$ and
\begin{equation*}
  A=\frac{3^d}{\epsilon_d}(\Norm{T}{L^q\to L^{q,\infty}}+\Norm{\mathcal M_{T,3}^{\#}}{L^r\to L^{r,\infty}}),
\end{equation*}
and hence
\begin{equation*}
  (3+c)A=C_d(\Norm{T}{L^q\to L^{q,\infty}}+\Norm{\mathcal M_{T,3}^{\#}}{L^r\to L^{r,\infty}}).
\end{equation*}
\end{proof}

An immediate consequence of the previous results is the following:

\begin{cor}\label{cor:CZXsparse}
Let $T\in\bddlin(L^2(\R^2))$ be an operator with a CZX kernel
satisfying $\theta_2 = 1$. Then for every $p>1$, the operator $T$ satisfies pointwise $L^p$-sparse domination with constants $C\lesssim p'$ and an absolute $\epsilon>0$.
\end{cor}

\begin{proof}
We know from Theorem \ref{thm:T1CZX} that $T$ is bounded from $L^1(\R^2)$ to $L^{1,\infty}(\R^2)$. By Lemma \ref{lem:LOmax}, we have that $\mathcal M_{T,3}^{\#}f\lesssim M_* f$. Since the strong maximal operator is bounded from $L^p(\R^2)$ to itself, and hence to $L^{p,\infty}(\R^2)$, the assumptions of Theorem \ref{thm:LO} are satisfied. Since $d=2$, it is immediate that $\epsilon=\epsilon_2$ provided by that theorem is absolute. In order to obtain the claim $C\lesssim p'$, we observe for completeness the following (probably well known) estimate for $M_*f\leq M^2 M^1 f$, where $M^i$ is the one-dimensional maximal operator with respect to the $i$th variable:
\begin{equation*}
\begin{split}
  \abs{\{M_*f>\lambda\}}
  &=\int_{\R}\abs{\{y:M_*f(x,y)>\lambda\}}\ud x
  \leq\int_{\R}\abs{\{y:M^2 M^1 f(x,y)>\lambda\}}\ud x \\
  &\leq\int_{\R} \lambda^{-p}(\Norm{M^2}{L^p\to L^{p,\infty}}\Norm{M^1 f(x,\cdot)}{L^p(\R)})^p\ud x \\
  &=\lambda^{-p}\Norm{M^2}{L^p\to L^{p,\infty}}^p\Norm{M^1 f}{L^p(\R^2)}^p \\
  &\leq \lambda^{-p}\Norm{M^2}{L^p\to L^{p,\infty}}^p(\Norm{M^1}{L^p\to L^p} \Norm{f}{L^p(\R^2)})^p.
\end{split}
\end{equation*}
It follows that
\begin{equation*}
  \Norm{M_*}{L^p(\R^2)\to L^{p,\infty}(\R^2)}\leq \Norm{M^2}{L^p\to L^{p,\infty}} \Norm{M^1}{L^p\to L^p}
  \lesssim 1\cdot p',
\end{equation*}
using the fact that the $L^p$ norm of the usual maximal operator is $O(p')$, while its $L^p$-to-$L^{p,\infty}$ norm can be estimated independently of $p$. In fact,
\begin{equation*}
\begin{split}
  \abs{\{Mf>\lambda\}}
  =\abs{\{(Mf)^p>\lambda^p\}}
  &\leq\abs{\{M(\abs{f}^p)>\lambda^p\}} \\
  &\leq\lambda^{-p}\Norm{M}{L^1\to L^{1,\infty}}\Norm{\abs{f}^p}{L^1}
  =\lambda^{-p}\Norm{M}{L^1\to L^{1,\infty}}\Norm{f}{L^p}^p,
\end{split}
\end{equation*}
and hence
\begin{equation*}
  \Norm{M}{L^p\to L^{p,\infty}}\leq\Norm{M}{L^1\to L^{1,\infty}}^{1/p}\leq\Norm{M}{L^1\to L^{1,\infty}}.
\end{equation*}
\end{proof}

\begin{cor}\label{cor:CZX-Ap}
Let $T\in\bddlin(L^2(\R^2))$ be an operator with a CZX kernel satisfying $\theta_2 = 1$. Then for every $p\in(1,\infty)$ and every $w\in A_p(\R^2)$, the operator $T$ extends boundedly to $L^p(w)$ with norm
\begin{equation*}
  \Norm{T}{\bddlin(L^p(w))}\lesssim_p [w]_{A_p}^{p'}.
\end{equation*}
\end{cor}

\begin{proof}
This follows by the same reasoning as \cite[Theorem 1.6]{LPRR}. The said theorem is stated for different operators, but its proof only uses a certain sparse domination estimate for these operators, which is a slightly weaker variant (so-called form domination) of what we proved (pointwise domination) for operators with CZX kernel in Corollary \ref{cor:CZXsparse}, and hence the same reasoning applies to the case at hand.
\end{proof}
A curious feature of the above proof of the weighted estimates is that it passes through estimates involving the
  strong maximal operator, which in principle should be forbidden in the theory of standard $A_p$ weights;
  indeed, the strong maximal operator is bounded in $L^p(w)$ for strong $A_p$ weights only.
  The resolution of this paradox is that we only use the strong maximal operator as an intermediate step,
  in a part of the argument with no weights yet present, to establish some sparse bounds, which in turn imply the weighted estimates.

To conclude this section, we discuss commutator estimates. In fact, combining the ideas from  \cite{LOR1} and \cite{LO} we can establish the following sparse domination principle.
\begin{prop}\label{prop:abstractcommu}
Let $T$ be a linear operator that is bounded from $L^q(\R^d)$ to $L^{q,\infty}(\R^d)$, and such that $\mathcal M_{T,3}^{\#}$ is bounded from $L^r(\R^d)$ to $L^{r,\infty}(\R^d)$ for some $1\leq q,r<\infty$. Let $s=\max(q,r)$. Then there exists an $\epsilon$-sparse family $\mathscr S$ such that the commutator $[b,T]f:=bTf-T(bf)$ satisfies 
\[
|[b,T]f|\le C \left(\sum_{S\in \mathscr S} \ave{|(b-\ave{b}_S)f|}_{S,s}1_S+ \sum_{S\in \mathscr S} |b-\ave{b}_S| \ave{|f|}_{S,s}1_S\right),
\]
with constants
\begin{equation*}
   C=c_{d}(\Norm{T}{L^q\to L^{q,\infty}}+\Norm{\mathcal M_{T,3}^{\#}}{L^r\to L^{r,\infty}}),\qquad \epsilon=\epsilon_d.
\end{equation*}

\end{prop}

\begin{proof}The proof is actually similar to \cite[Theorem 1.1]{LO} but one should adapt ideas used in \cite{LOR1}. Let $c, A,\epsilon_d$ be the same as those in the proof of Theorem \ref{thm:LO}. 
 Apart from the sets 
 \[
   \{M(\abs{f}^s1_{3Q})^{1/s}>c\ave{|f|}_{3Q,s}\},\quad\{\abs{T(f1_{3Q})}>A\ave{|f|}_{3Q,s}\},\quad\{\mathcal M_{T,3}^{\#}(f1_{3Q})>A\ave{|f|}_{3Q,s}\},
 \]
 for the same reason, we can also let each of the sets
  \begin{align*}
    \{M(\abs{(b-\ave{b}_{3Q})f}^s1_{3Q})^{1/s}&>c\ave{\abs{(b-\ave{b}_{3Q})f}}_{3Q,s}\},\\
    \{\abs{T((b-\ave{b}_{3Q})f1_{3Q})}&>A\ave{\abs{(b-\ave{b}_{3Q})f}}_{3Q,s}\},\\
    \{\mathcal M_{T,3}^{\#}((b-\ave{b}_{3Q})f1_{3Q})&>A\ave{\abs{(b-\ave{b}_{3Q})f}}_{3Q,s}\}
  \end{align*}
  have measure at most $\epsilon_d\abs{Q}$. Denote the union of the above six sets by $E$ and set $\Omega=E\cap Q$.  Manipulating in the same way as in \cite[Theorem 1.1]{LO}, we get a family of pairwise disjoint cubes $\{P_j\}\subset Q$ such that 
$
  \sum_{j}|P_j|\le \frac 12 |Q|
$ and $|\Omega \setminus \cup_j P_j |=0$. The latter implies that for a.e. $x\in Q\setminus \cup_j P_j$ we have
\begin{align*}
|T(f1_{3Q})(x)|\le A\ave{|f|}_{3Q,s}, \quad |T((b-\ave{b}_{3Q})f1_{3Q})(x)|\le A\ave{\abs{(b-\ave{b}_{3Q})f}}_{3Q,s}.
\end{align*}On the other hand, similarly as in \cite[Theorem 1.1]{LO} we also have for a.e. $x\in P_j$ that
\begin{align*}
|T(f1_{3Q\setminus 3P_j})(x)|&\le (2+c) A \ave{|f|}_{3Q,s},\\ |T((b-\ave{b}_{3Q})f1_{3Q\setminus 3P_j})(x)|&\le (2+c)A\ave{\abs{(b-\ave{b}_{3Q})f}}_{3Q,s}.
\end{align*}
Thus 
\begin{align*}
|[b,T](f1_{3Q})|1_Q(x)&= |[b,T](f1_{3Q})|1_{Q\setminus \cup_j P_j}(x)+\sum_j  |[b,T](f1_{3Q\setminus 3P_j})|1_{P_j}\\
&\hspace{3cm}+\sum_j  |[b,T](f1_{3P_j})|1_{P_j}\\
&=|[b-\ave{b}_{3Q},T](f1_{3Q})|1_{Q\setminus \cup_j P_j}(x)+\sum_j  |[b-\ave{b}_{3Q},T](f1_{3Q\setminus 3P_j})|1_{P_j}\\
&\hspace{3cm}+\sum_j  |[b,T](f1_{3P_j})|1_{P_j}\\
&\le (3+c)A|b-\ave{b}_{3Q}| \ave{|f|}_{3Q,s}1_Q+(3+c)A \ave{\abs{(b-\ave{b}_{3Q})f}}_{3Q,s}1_Q\\
&\hspace{3cm}+\sum_j  |[b,T](f1_{3P_j})|1_{P_j}.
\end{align*}Note that the linearity of $T$ is used in the second step. 
With the recursive inequality at hand, the rest is standard (see e.g. \cite[Lemma 2.1]{LO}). And since the constants $c$ and $A$ are the same as those in Theorem \ref{thm:LO}, we get the desired
constant in the sparse domination.
\end{proof}
Analogous to Corollary \ref{cor:CZXsparse}, we have the following result
\begin{prop}\label{prop:sparsecommu}
Let $T\in\bddlin(L^2(\R^2))$ be an operator with a CZX kernel satisfying $\theta_2 = 1$. Then for every $p>1$, the commutator $[b,T]$ satisfies 
\[
|[b,T]f|\le C p'\left(\sum_{S\in \mathscr S} \ave{|(b-\ave{b}_S)f|}_{S,p}1_S+ \sum_{S\in \mathscr S} |b-\ave{b}_S| \ave{|f|}_{S,p}1_S\right),
\]where $\mathscr S$ is an $\epsilon$-sparse family with $\epsilon > 0$ absolute and $C$ is an absolute constant depending only on the operator $T$. 
\end{prop}

In \cite[Theorem 5.2]{L} a two-weight commutator estimate for rough homogeneous singular integrals was formulated. As our sparse forms are the same as there,
the two-weight commutator estimate of Theorem \ref{thm:quantitative} follows as a direct consequence of Proposition \ref{prop:sparsecommu}.

\end{document}